\documentclass[reqno]{amsart}

\usepackage[T1]{fontenc}
\usepackage[utf8]{inputenc}

\usepackage{lmodern}
\usepackage{microtype}

\usepackage{amsfonts,amssymb,amsthm,amsmath,graphicx,amscd,mathtools,mathrsfs}
\usepackage[shortlabels]{enumitem}

\usepackage[dvipsnames,svgnames,x11names]{xcolor}
\usepackage[all]{xy}
\usepackage[normalem]{ulem}

\usepackage{hyperref}
\hypersetup{
  colorlinks=true,
  linkcolor=MidnightBlue,
  citecolor=ForestGreen,
  urlcolor=BrickRed
}

\newtheorem{theor}{Theorem}[section]

\newtheorem{thm}[theor]{Theorem}
\newtheorem{prop}[theor]{Proposition}

\newtheorem{lem}[theor]{Lemma}

\newtheorem{cor}[theor]{Corollary}

\newtheorem{thmx}{Theorem}

\theoremstyle{definition}

\newtheorem{remark}[theor]{Remark}

\newcommand{\B}{\mathbb{B}}
\newcommand{\N}{\mathbb{N}}

\newcommand{\R}{\mathbb{R}}
\newcommand{\C}{\mathbb{C}}

\newcommand{\de}{\partial}
\newcommand{\deb}{\bar\partial}
\newcommand{\W}{\Omega}
\newcommand{\w}{\omega}
\newcommand{\Vol}{\operatorname{Vol}}

\newcommand{\lmb}{\lambda}

\newcommand{\Aut}{\operatorname{Aut}}

\newcommand{\Ric}{\operatorname{Ric}}
\newcommand{\D}{D^{g_\Omega}}

\newcommand{\ov}[1]{\overline{#1}}

\title[Bergman--Einstein Rigidity for Hartogs Domains]{Bergman--Einstein Rigidity for Hartogs Domains over Bounded Homogeneous Domains}

\author{Roberto Mossa}
\address{Dipartimento di Matematica e Informatica, Universit\`a degli Studi di Cagliari,
Via Ospedale 72, 09124 Cagliari, Italy}
        \email{roberto.mossa@unica.it}
        
        \thanks{
The author gratefully acknowledges support from INdAM and GNSAGA -- Gruppo Nazionale per le Strutture Algebriche, Geometriche e le loro Applicazioni, and from ProBiKi of Fondazione di Sardegna (Italy), CUP F83C26000350007.}

\subjclass[2020]{Primary 32A25, 32Q20; Secondary 32M10, 32Q15, 53C55}

\keywords{Bergman kernel, Bergman metric, Kähler--Einstein metric, Hartogs domain, bounded homogeneous domain}

\date{}

\begin{document}

\begin{abstract}
We prove a rigidity theorem for the Bergman metric on Hartogs domains over bounded
homogeneous domains. Let $\Omega\subset \mathbb C^n$ be a bounded homogeneous domain, let
$K_\Omega$ denote its Bergman kernel, and consider
$$
\Omega_{m,s}:=\{(z,\zeta)\in \Omega\times \mathbb C^m:\ \|\zeta\|^2<K_\Omega(z,\bar z)^{-s}\},
\qquad m\ge 1,\quad s>0.
$$
We prove that the Bergman metric of $\Omega_{m,s}$ is K\"ahler--Einstein if and only if
$\Omega_{m,s}$ is homogeneous. These conditions are equivalent to $\Omega_{m,s}$ being
biholomorphic to $\mathbb B^{n+m}$, and they hold precisely when $\Omega$ is biholomorphic
to $\mathbb B^n$ and $s=1/(n+1)$.

This gives a positive answer to Yau's question within this class and may be viewed as a
Cheng-type rigidity phenomenon beyond the smoothly bounded strictly pseudoconvex setting. The
proof combines the explicit formula for the Bergman kernel of $\Omega_{m,s}$ with the
structural invariants of the bounded homogeneous base.
\end{abstract}

\maketitle

\tableofcontents

\section{Introduction}

The Bergman metric is one of the canonical K\"ahler metrics associated with a
bounded domain. A basic problem in several complex variables and complex differential geometry is
to understand how curvature properties of the Bergman metric constrain the complex geometry of
the underlying domain. A guiding question in this direction, going back to Yau
\cite{Yau1982ProblemSection}, is the following.

\vskip0.1cm

\noindent\textbf{Question.}
Let $D$ be a bounded domain and let $g_D$ be its Bergman metric, assumed to be complete. If
$g_D$ is K\"ahler--Einstein, must $D$ be a bounded homogeneous domain?

\vskip0.1cm

The question is natural because the Bergman metric of a bounded homogeneous domain is
K\"ahler--Einstein. Thus Yau's question asks whether, in this context, the converse implication
holds.

For strictly pseudoconvex domains with smooth boundary, Yau's question admits a positive
answer, following the resolution of a conjecture formulated by Cheng. Cheng conjectured that if
$D\subset \mathbb C^n$ is a bounded strictly pseudoconvex domain with smooth boundary and
its Bergman metric is K\"ahler--Einstein, then $D$ is biholomorphic to the unit ball.
In dimension two, Fu and Wong \cite{FuWong2005StrictlyPseudoconvexDomains} proved this
for simply connected domains, and Nemirovski and Shafikov
\cite{NemirovskiiShafikov2006ConjectureCheng} removed the simple connectivity assumption.
Huang and Xiao \cite{HuangXiao2021BergmanEinstein} proved the conjecture in dimensions $n\ge 3$.
It is also worth recalling that any bounded homogeneous domain with $C^2$ boundary is
biholomorphic to the unit ball, as shown independently by Wong and Rosay
\cite{Wong1977CharacterizationUnitBall,Rosay1979SurUneCaracterisation}. Therefore, in the
smoothly bounded strictly pseudoconvex setting, the K\"ahler--Einstein condition on the Bergman
metric is so rigid that it forces the domain to be biholomorphic to the unit ball.

The purpose of the present paper is to establish an analogous rigidity phenomenon in a different
and highly structured setting, namely that of Hartogs domains over bounded homogeneous bases.
These domains are in general neither smoothly bounded nor homogeneous, and therefore provide a
natural testing ground for Yau's question and for Cheng-type rigidity outside the strictly
pseudoconvex framework.

Let $\Omega\subset \mathbb C^n$ be a bounded homogeneous domain, let $K_\Omega$ be its Bergman
kernel, and consider the Hartogs domain
$$
\Omega_{m,s}:=
\{(z,\zeta)\in \Omega\times \mathbb C^m:\ \|\zeta\|^2<K_\Omega(z,\bar z)^{-s}\},
\qquad
m\in\mathbb N^*,
\qquad
s>0.
$$
Since the constant function belongs to the Bergman space of $\Omega$, we have
$K_\Omega(z,\bar z)\ge 1/\Vol(\Omega)$. Thus $\Omega_{m,s}$ is bounded.

Hartogs domains over bounded homogeneous bases have been studied from several points of view; see,
for instance,
\cite{IshiParkYamamori2017BergmanKernel,Seo2018BiholomorphismsHartogsSiegel,PanWangZhang2016KahlerEinstein}.
The class considered here contains, as a special case, Hartogs domains over bounded symmetric
domains, which have been extensively investigated in the literature; see, e.g.,
\cite{WangYinZhangRoos2006KahlerEinstein,YinLuRoos2004NewClasses,LoiMossaZuddas2025BergmanMetric,MossaZedda2022SymplecticGeometry}
and the references therein.

Our main result gives a complete rigidity theorem for this class. It shows that, among the
domains $\Omega_{m,s}$ with $s>0$, the K\"ahler--Einstein condition for the Bergman metric,
homogeneity, and biholomorphic equivalence to the unit ball are all equivalent.

\begin{thmx}\label{main}
Let $\Omega \subset \mathbb C^n$ be a bounded homogeneous domain, and let
$\Omega_{m,s}$ be the Hartogs domain defined above. Assume that $s>0$.
Then the following conditions are equivalent:
\begin{enumerate}
    \item\label{main:ke} the Bergman metric of $\Omega_{m,s}$ is K\"ahler--Einstein;
    \item\label{main:hom} $\Omega_{m,s}$ is homogeneous;
    \item\label{main:ball} $\Omega_{m,s}$ is biholomorphic to $\mathbb B^{n+m}$;
    \item\label{main:base} $\Omega$ is biholomorphic to $\mathbb B^n$ and $s=\frac1{n+1}$.
\end{enumerate}
\end{thmx}

Thus, in the family of Hartogs domains over bounded homogeneous bases, Yau's question has a
positive answer in the strongest possible form: if the Bergman metric of $\Omega_{m,s}$ is
K\"ahler--Einstein and $s>0$, then $\Omega_{m,s}$ is not merely homogeneous, but
biholomorphic to the unit ball. Moreover, this happens only in the expected trivial case: the
base is itself the unit ball and the exponent is forced to be $s=\frac1{n+1}$.

For the K\"ahler--Einstein aspect, Theorem~\ref{main} extends the corresponding ball
characterization obtained for Cartan--Hartogs domains in
\cite[Theorem~1.2]{LoiMossaZuddas2025BergmanMetric} to the broader setting of Hartogs domains
over bounded homogeneous bases. When $\Omega=\mathbb B^n$, write $\kappa_n=n!/\pi^n$.
Since $K_{\mathbb B^n}(z,\bar z)=\kappa_n(1-\|z\|^2)^{-(n+1)}$, the change of variables
$w=\kappa_n^{s/2}\zeta$ identifies $\Omega_{m,s}$ with the egg domain
\[
E_p=\{(z,w)\in\C^n\times\C^m:\ \|z\|^2+\|w\|^{2p}<1\},
\qquad p=\frac{1}{(n+1)s}.
\]
Ebenfelt, Xiao and Xu \cite[Proposition~1.10]{EXX24} proved that the Bergman metric
of $E_p$ is K\"ahler--Einstein precisely when $p=1$. This proves Theorem~\ref{main} when $\Omega=\mathbb B^n$,
for every $m\ge 1$ and $s>0$. For earlier results in fiber dimension one,
see \cite{FuWong2005StrictlyPseudoconvexDomains,Cho21,EXX23} and \cite[Remark~1.11]{EXX24}.

The equivalence with homogeneity in Theorem~\ref{main} may also be viewed as a rigidity
statement for the automorphism geometry of the domains $\Omega_{m,s}$. For $s>0$,
homogeneity of $\Omega_{m,s}$ forces $\Omega\cong \mathbb B^n$ and $s=\frac1{n+1}$. For Cartan--Hartogs
domains, analogous homogeneity characterizations are known; see
\cite[Theorem~1.2]{LoiMossaZuddas2025BergmanMetric}. Earlier references trace this phenomenon
back to unpublished communications of G.~Roos, as recorded for instance in
\cite[Lemma~3.1]{AhnByunPark2012AutomorphismsHartogsDomain}.

\begin{remark}\label{rem:szero}
The endpoint $s=0$ gives
$$
\Omega_{m,0}=\Omega\times \mathbb B^m,
$$
and the Bergman metric of $\Omega_{m,0}$ is the product of the Bergman metrics of $\Omega$
and $\mathbb B^m$. Since both factors are K\"ahler--Einstein with Einstein constant $-1$,
the product metric is again K\"ahler--Einstein. In general, however,
$\Omega\times \mathbb B^m$ is not biholomorphic to $\mathbb B^{n+m}$. Thus the conclusion
of Theorem~\ref{main} fails when $s=0$.
\end{remark}

A related rigidity result was recently obtained by Palmieri
\cite{Palmieri2025BergmanMetricsInducedBall}, who proved that, among Hartogs domains over
bounded homogeneous bases, the ball is also singled out by the requirement that the Bergman
metric be, up to rescaling, induced by the Bergman metric of a finite-dimensional unit ball.

We briefly describe the strategy of the proof of Theorem \ref{main}. The first ingredient is the explicit formula of
Ishi--Park--Yamamori for the Bergman kernel of $\Omega_{m,s}$
\cite{IshiParkYamamori2017BergmanKernel}, which makes it possible to reduce the Einstein
equation for the Bergman metric of the total space to a scalar identity in one radial variable.
The second ingredient is the structure theory of bounded homogeneous domains, encoded in the
polynomial
$$
F_\Omega(\sigma)=\prod_{k=1}^r\prod_{i=1}^{1+p_k+b_k}
\left(1+\frac{\sigma}{a_{k,i}}\right),
$$
whose zeros reflect the intrinsic invariants of the base. The interaction between these two
ingredients forces the radial part of the Bergman kernel to collapse to a pure pole. From this
collapse one recovers both the critical value $s=\frac1{n+1}$ and the rank-one condition on
$\Omega$, hence the ball characterization.

The soliton version of the ball characterization is known for Cartan--Hartogs domains
\cite[Theorem~1.2]{LoiMossaZuddas2025BergmanMetric}. For a related rigidity result for
solitons induced by homogeneous bounded domains, see \cite{LoiMossa2023HolomorphicIsometries}.
Sha \cite[Theorem~4.2]{Sha26} has shown that a complete Bergman metric on a bounded domain
that is a K\"ahler--Ricci soliton is K\"ahler--Einstein. This gives the following consequence
of Theorem~\ref{main}.

\begin{cor}\label{cor:krs}
Let $\Omega\subset\C^n$ be a bounded homogeneous domain, $m\ge 1$ and $s>0$.
Assume that the Bergman metric $g_{\Omega_{m,s}}$, with K\"ahler form $\omega_{\Omega_{m,s}}$,
is a K\"ahler--Ricci soliton, that is,
\[
\Ric(\omega_{\Omega_{m,s}})+\mathcal L_X\omega_{\Omega_{m,s}}
=\lambda\,\omega_{\Omega_{m,s}}
\]
for some real holomorphic vector field $X$ on $\Omega_{m,s}$ and some $\lambda\in\R$.
Then $\Omega_{m,s}$ is biholomorphic to $\B^{n+m}$.
\end{cor}

\begin{proof}
The base $\Omega$ admits a continuous negative plurisubharmonic exhaustion $\rho$
\cite[p.~161]{KaiOhsawa07}. Set $t(z,\zeta)=K_\Omega(z,\bar z)^s\|\zeta\|^2$.
The function $t$ is plurisubharmonic, since
$\log t=s\log K_\Omega+\log\|\zeta\|^2$ is plurisubharmonic, with value $-\infty$
on the zero section. Thus $\max\{\rho(z),t(z,\zeta)-1\}$ is a continuous negative
plurisubharmonic exhaustion of $\Omega_{m,s}$. Indeed, on each negative sublevel
set, $z$ stays in a compact subset of $\Omega$ and $t$ stays uniformly below $1$.
The bounded hyperconvex domain $\Omega_{m,s}$ is therefore Bergman complete
by \cite{BlockiPflug98}. Sha's theorem gives the Einstein condition, and
Theorem~\ref{main} gives the conclusion.
\end{proof}

The next section proves Theorem~\ref{main}. We first recall the necessary structural notation
and the Bergman kernel formula.

\section*{AI Declaration}
I used ChatGPT (OpenAI, [Sol 6.1]) to check the consistency of mathematical arguments
and to improve the English and exposition of the manuscript. I independently checked
the mathematical arguments and references incorporated into the paper and take full
responsibility for its contents.


\section{Proof of Theorem~\ref{main}}\label{sec:proof}

\subsection{Preliminaries}

We begin by recalling the structural notation and the explicit Bergman kernel formula needed in the proof of Theorem~\ref{main}.

Let $\W\subset \C^n$ be a bounded homogeneous domain. Following
\cite[Proposition~2.2, Proposition~2.3, equation~(9), Theorem~2.4]{IshiParkYamamori2017BergmanKernel},
one associates to $\W$ a rank $r\in\N^*$ and integers
$$
p_k,\ q_k,\ b_k\in\N,
\qquad
1\leq k\leq r.
$$
For each $k$ and each index $1\le i\le 1+p_k+b_k$, one defines
\begin{equation}\label{aki}
a_{k,i}:=\frac{i+\frac{q_k}{2}}{2+p_k+q_k+b_k}.
\end{equation}
Set $C_\W:=\min_{k,i}a_{k,i}>0$.
The associated structural polynomial is
\begin{equation*}\label{Fom}
F_\W(\sigma):=
\prod_{k=1}^r\prod_{i=1}^{1+p_k+b_k}
\left(1+\frac{\sigma}{a_{k,i}}\right).
\end{equation*}
Its degree is
\begin{equation}\label{deg}
\deg F_\W=\sum_{k=1}^r(1+p_k+b_k)=\dim_\C\W=n.
\end{equation}

For $m\ge1$ and $s>0$, recall that
\begin{equation*}\label{dom}
\W_{m,s}:=
\{(z,\zeta)\in \W\times \C^m:\ \|\zeta\|^2<K_\W(z,\ov z)^{-s}\}.
\end{equation*}
In the Hartogs expansion \cite[Proposition~4.1]{IshiParkYamamori2017BergmanKernel},
the term of fiber degree $k$ involves the reproducing kernel of
$A^2(\W,K_\W^{-\sigma_k}dV)$ with $\sigma_k=(k+m)s$. For $\sigma>-C_\W$,
the diagonal weighted kernel is $F_\W(\sigma)K_\W(z,\bar z)^{1+\sigma}$
\cite[Proposition~2.3 and Corollary~2.10]{IshiParkYamamori2017BergmanKernel}.
Since $s>0$, every $\sigma_k$ lies in this range.

To express the resulting kernel in closed form, we use the Pochhammer symbol with the convention
\[
(x+1)_0:=1,\qquad (x+1)_j:=(x+1)(x+2)\cdots(x+j)\quad\text{for }j\ge1.
\]
Thus we write
\begin{equation}\label{poch}
F_\W(sx)=\sum_{j=0}^n c_j(s)(x+1)_j.
\end{equation}
Then \cite[Theorem~4.4]{IshiParkYamamori2017BergmanKernel} gives the diagonal Bergman kernel
of $\W_{m,s}$ in the form
\begin{equation}\label{kdiag}
K_{\W_{m,s}}(z,\zeta)
=
\frac{K_\W(z,\ov z)^{ms+1}}{\pi^m}
\sum_{j=0}^n
\frac{c_j(s)(j+m)!}{(1-t)^{j+m+1}},
\qquad
t:=K_\W(z,\ov z)^s\|\zeta\|^2.
\end{equation}

It is convenient to set
\begin{equation}\label{Adef}
A:=ms+1
\end{equation}
and
\begin{equation}\label{Rdef}
R(t):=
\sum_{j=0}^n \frac{A_j}{(1-t)^{j+m+1}},
\qquad
A_j:=c_j(s)(j+m)!.
\end{equation}
Then \eqref{kdiag} becomes
\begin{equation}\label{kR}
K_{\W_{m,s}}(z,\zeta)
=
\frac{K_\W(z,\ov z)^A}{\pi^m}R(t),
\qquad
t=K_\W(z,\ov z)^s\|\zeta\|^2.
\end{equation}
Accordingly,
\begin{equation}\label{pot}
\psi_{\W_{m,s}}(z,\zeta):=\log K_{\W_{m,s}}(z,\zeta)
=
A\log K_\W(z,\ov z)+\log R(t)-m\log\pi
\end{equation}
is a real-analytic K\"ahler potential for the Bergman metric $g_{\W_{m,s}}$.

We shall also use the following standard fact.

\begin{lem}\label{baseke}
Let $\W\subset \C^n$ be a bounded homogeneous domain, and let $g_\W$ be its Bergman metric.
Then $g_\W$ is K\"ahler--Einstein with Einstein constant $-1$. In particular, there exists
a positive constant $C_1$ such that
\begin{equation}\label{detbase}
\det g_\W(z)=C_1\,K_\W(z,\ov z).
\end{equation}
\end{lem}

\begin{proof}
It is well known that the Bergman metric of a bounded homogeneous domain is
K\"ahler--Einstein with Einstein constant \(-1\); see
\cite[Theorem~4.1]{Kobayashi1959GeometryBounded} and
\cite{Kaneyuki1971BOOKHomogeneousBoundedDomains}. In particular,
$
\Ric(g_\W)=-\w_\W.
$
Since
$
\Ric(g_\W)=-\frac{i}{2}\de\deb\log\det g_\W
$
and
$
\w_\W=\frac{i}{2}\de\deb\log K_\W,
$
it follows that \(\de\deb\log(\det g_\W/K_\W)=0\). The ratio \(\det g_\W/K_\W\) is invariant
under biholomorphisms, hence in particular under \(\Aut(\W)\). As \(\W\) is homogeneous, this
ratio is constant, and \eqref{detbase} follows.
\end{proof}

\subsection{The Einstein identity and the determinant formula}

\begin{lem}\label{ein}
Assume that \(g_{\W_{m,s}}\) is K\"ahler--Einstein with Einstein constant \(\lmb\). Then, on
every simply connected coordinate neighborhood \(U\subset \W_{m,s}\), there exists a holomorphic
function \(f\in\mathcal O(U)\) such that
\begin{equation}\label{localein}
\det(g_{\alpha\ov\beta})=e^{f+\ov f}K_{\W_{m,s}}^{-\lmb}
\qquad\text{on }U.
\end{equation}
\end{lem}

\begin{proof}
The Einstein equation reads $\Ric(g_{\W_{m,s}})=\lmb\,\w_{\W_{m,s}}$. In local coordinates,
\[
-\frac{i}{2}\de\deb\log\det(g_{\alpha\ov\beta})
=
\lmb\,\frac{i}{2}\de\deb\log K_{\W_{m,s}},
\]
because $\log K_{\W_{m,s}}$ is a local K\"ahler potential of the Bergman metric. Hence
\[
\de\deb\bigl(\log\det(g_{\alpha\ov\beta})+\lmb\log K_{\W_{m,s}}\bigr)=0.
\]
On a simply connected neighborhood, the real-valued function in parentheses is pluriharmonic,
hence the real part of a holomorphic function $2f$. Exponentiating gives \eqref{localein}.
\end{proof}

\begin{lem}\label{lmb}
Under the assumptions of Lemma~\ref{ein}, one has $\lmb=-1$.
\end{lem}

\begin{proof}
Restrict \eqref{localein} to the zero section
\[
\Sigma:=\{(z,0)\in\W_{m,s}: z\in\W\}.
\]
Since $t=K_\W(z,\ov z)^s\|\zeta\|^2$, one has $t=0$ on $\Sigma$, and \eqref{kR} yields
\begin{equation}\label{kzero}
K_{\W_{m,s}}(z,0)=\frac{K_\W(z,\ov z)^A}{\pi^m}R(0).
\end{equation}

Now compute the metric tensor on $\Sigma$. We use the convention that
$i,j\in\{1,\dots,n\}$ denote base indices, whereas
$\mu,\nu\in\{1,\dots,m\}$ denote fiber indices. Set
\begin{equation*}\label{pdef}
p(t):=\frac{R'(t)}{R(t)}.
\end{equation*}
Since $t$ contains the factor $\|\zeta\|^2$, the mixed base-fiber block vanishes at $\zeta=0$.
Moreover, by differentiating the potential \eqref{pot}, and using the definition of $p(t)$ above,
\[
(g_{\W_{m,s}})_{i\ov j}(z,0)=A(g_\W)_{i\ov j}(z),
\qquad
(g_{\W_{m,s}})_{\mu\ov\nu}(z,0)=p(0)K_\W(z,\ov z)^s\delta_{\mu\nu}.
\]

Hence
\[
g_{\W_{m,s}}(z,0)=A\,g_\W(z)\oplus p(0)K_\W(z,\ov z)^s I_m,
\]
and therefore
\[
\det(g_{\alpha\ov\beta})(z,0)
=
A^n p(0)^m K_\W(z,\ov z)^{ms}\det g_\W(z).
\]
Using \eqref{detbase} and the definition of $A$ in \eqref{Adef}, we obtain 
\begin{equation}\label{detzero2}
\det(g_{\alpha\ov\beta})(z,0)=C_2K_\W(z,\ov z)^A.
\end{equation}

Restricting \eqref{localein} to $\Sigma$ and comparing with \eqref{kzero} and \eqref{detzero2},
we find
\[
e^{h+\ov h}=C_3K_\W(z,\ov z)^{A(1+\lmb)}
\]
for a holomorphic function $h$ on the base chart. Applying $\de\deb\log$ to both sides, we obtain
\[
0=A(1+\lmb)\,\de\deb\log K_\W(z,\ov z).
\]

Since
$
(g_{\W_{m,s}})_{i\ov j}(z,0)=A(g_\W)_{i\ov j}(z)
$
and $g_{\W_{m,s}}$ is positive definite, necessarily $A>0$. Moreover,
$\frac{i}{2}\de\deb\log K_\W$ is the K\"ahler form of the Bergman metric of $\W$, hence it is
not identically zero. Therefore $1+\lmb=0$, that is, $\lmb=-1$.
\end{proof}
From now on, the local Einstein identity becomes
\begin{equation}\label{localein2}
\det(g_{\alpha\ov\beta})=e^{f+\ov f}K_{\W_{m,s}}.
\end{equation}

To compute the determinant away from the zero section, we use the diastasis of the base.

\begin{lem}\label{lem:diastasis}
Let $\Omega\subset \C^n$ be a bounded domain, let $K_\Omega(z,\bar w)$ be its Bergman kernel,
and fix $z_0\in\Omega$. Since $K_\Omega(z_0,\bar z_0)>0$, there is a simply connected
neighborhood $U$ of $z_0$ on which $K_\Omega(z,\bar z_0)\neq0$. On $U$, Calabi's diastasis
of the Bergman metric centered at $z_0$ \cite{Calabi1953IsometricImbedding} is given by
\[
\D_{z_0}(z,\bar z):=
\log\left(
\frac{K_\Omega(z_0,\bar z_0)K_\Omega(z,\bar z)}
{K_\Omega(z,\bar z_0)K_\Omega(z_0,\bar z)}
\right).
\]
Then
\[
\frac{\partial \D_{z_0}}{\partial z_i}(z_0,\bar z_0)=0,
\qquad
\frac{\partial \D_{z_0}}{\partial \bar z_j}(z_0,\bar z_0)=0,
\]
and
\[
\frac{\partial^2 \D_{z_0}}{\partial z_i\partial \bar z_j}(z_0,\bar z_0)
=
\frac{\partial^2\log K_\Omega}{\partial z_i\partial \bar z_j}(z_0,\bar z_0).
\]
\end{lem}

\begin{proof}
Expand the definition of $\D_{z_0}$ and differentiate directly. The mixed derivatives of
$\log K_\Omega(z,\bar z_0)$ and $\log K_\Omega(z_0,\bar z)$ vanish because these terms are,
respectively, holomorphic and antiholomorphic in $z$.
\end{proof}

\begin{lem}\label{det}
Fix $z_0\in\Omega$. Then, at every point $(z_0,\zeta)\in\Omega_{m,s}$,
\begin{equation}\label{detfull}
\det(g_{\alpha\bar\beta})(z_0,\zeta)
=
C_1K_\Omega(z_0,\bar z_0)^A
\bigl(A+s\,t\,p(t)\bigr)^n
p(t)^{m-1}\bigl(p(t)+t\,p'(t)\bigr),
\end{equation}
where
\[
A=ms+1,
\qquad
t=K_\Omega(z_0,\bar z_0)^s\|\zeta\|^2,
\qquad
p(t)=\frac{R'(t)}{R(t)}.
\]
\end{lem}

\begin{proof}
Fix $z_0\in\Omega$ and let $\D_{z_0}$ be the diastasis centered at $z_0$. Choose a simply
connected neighborhood of $z_0$ on which $K_\Omega(z,\bar z_0)\neq0$, and set
\[
h(z):=\frac{K_\Omega(z,\bar z_0)^s}{K_\Omega(z_0,\bar z_0)^{s/2}}.
\]
Introduce the new fiber variable $\eta=h(z)\zeta$. Then
\[
e^{s\D_{z_0}(z,\bar z)}\|\eta\|^2=K_\Omega(z,\bar z)^s\|\zeta\|^2.
\]
Thus, in the coordinates $(z,\eta)$, the variable $t$ defined in \eqref{kR} is given by
\[
t=e^{s\D_{z_0}(z,\bar z)}\|\eta\|^2.
\]

Since
\[
\log K_\Omega(z,\bar z)=\D_{z_0}(z,\bar z)
+\log|K_\Omega(z,\bar z_0)|^2-\log K_\Omega(z_0,\bar z_0),
\]
the potential \eqref{pot} differs from
\[
\widetilde\Phi(z,\eta):=
A\D_{z_0}(z,\bar z)+\log R\bigl(e^{s\D_{z_0}(z,\bar z)}\|\eta\|^2\bigr)
\]
only by a constant plus the real part of a holomorphic function. Hence both define the same
K\"ahler metric.

At the point $(z_0,\eta)$ one has $t=\|\eta\|^2$, the first derivatives of $t$ with respect to the
base variables vanish, and so the mixed base-fiber block is zero. For the horizontal block,
\[
(g_{\Omega_{m,s}})_{i\bar j}(z_0,\eta)
=
A\frac{\partial^2\D_{z_0}}{\partial z_i\partial\bar z_j}(z_0,\bar z_0)
+
p(t)\frac{\partial^2 t}{\partial z_i\partial\bar z_j}(z_0,\eta).
\]
Since
\[
\frac{\partial^2 t}{\partial z_i\partial\bar z_j}(z_0,\eta)
=
s\,t\,\frac{\partial^2\D_{z_0}}{\partial z_i\partial\bar z_j}(z_0,\bar z_0),
\]
Lemma~\ref{lem:diastasis} yields
\begin{equation}\label{hblock-rewrite}
(g_{\Omega_{m,s}})_{i\bar j}(z_0,\eta)
=
\bigl(A+s\,t\,p(t)\bigr)(g_\Omega)_{i\bar j}(z_0).
\end{equation}

For the vertical block one has
\[
\frac{\partial t}{\partial\eta_\mu}=\bar\eta_\mu,
\qquad
\frac{\partial t}{\partial\bar\eta_\nu}=\eta_\nu,
\qquad
\frac{\partial^2 t}{\partial\eta_\mu\partial\bar\eta_\nu}=\delta_{\mu\nu},
\]
so
\[
(g_{\Omega_{m,s}})_{\mu\bar\nu}(z_0,\eta)
=
p'(t)\bar\eta_\mu\eta_\nu+p(t)\delta_{\mu\nu}.
\]
Thus the vertical block is
\[
p(t)I_m+p'(t)\bigl(\bar\eta_\mu\eta_\nu\bigr)_{\mu,\nu},
\]
and by a straightforward computation, after applying a unitary change of coordinates sending $\eta$ to $(\sqrt t,0,\dots,0)$, one obtains
\[
\det\bigl(pI_m+p'(t)\bigl(\bar\eta_\mu\eta_\nu\bigr)_{\mu,\nu}\bigr)
=
p(t)^{m-1}\bigl(p(t)+t\,p'(t)\bigr).
\]

Multiplying the determinants of the two diagonal blocks and using \eqref{detbase}, we obtain
\[
\det(g_{\alpha\bar\beta})(z_0,\eta)
=
C_1K_\Omega(z_0,\bar z_0)
\bigl(A+s\,t\,p(t)\bigr)^n
p(t)^{m-1}\bigl(p(t)+t\,p'(t)\bigr).
\]
Finally, the holomorphic change of variables $(z,\zeta)\mapsto(z,\eta)$ contributes the factor
$|h(z_0)|^{2m}=K_\Omega(z_0,\bar z_0)^{sm}$, and \eqref{detfull} follows.
\end{proof}

\subsection{The one-variable equation}

\begin{lem}\label{lem:radialpluri}
Let $B\subset \C^m$ be a ball centered at the origin, and let $u\in C^\infty(B,\R)$.
Assume that
$$
u(\eta)=\phi(\|\eta\|^2)
$$
for some smooth function $\phi$ on $[0,\rho)$, where $\rho>0$ is such that
$$
B=\{\eta\in\C^m:\ \|\eta\|^2<\rho\}.
$$
If $u$ is pluriharmonic, then $u$ is constant.
\end{lem}
\begin{proof}
Set $r=\|\eta\|^2$. Then $u(\eta)=\phi(r)$ and
$$
\frac{\partial^2u}{\partial\eta_\mu\partial\bar\eta_\nu}
=
\phi''(r)\eta_\nu\bar\eta_\mu+\phi'(r)\delta_{\mu\nu}.
$$
Since $u$ is pluriharmonic, these mixed derivatives vanish identically.

If $m>1$, choose $\mu\neq \nu$. Then
$$
\phi''(r)\eta_\nu\bar\eta_\mu=0
\qquad
\text{for all }\eta\in B.
$$
Fix $r\in(0,\rho)$. Since $m>1$, one can choose $\eta\in B$ with $\|\eta\|^2=r$ and
$\eta_\nu\bar\eta_\mu\neq 0$, so $\phi''(r)=0$. Hence $\phi''=0$ on $(0,\rho)$. Taking now
$\mu=\nu$, we get $\phi'(r)=0$ for all $r\in(0,\rho)$. Therefore $\phi$ is constant on
$(0,\rho)$, hence on $[0,\rho)$ by continuity.

If $m=1$, the pluriharmonicity condition reads
$$
r\phi''(r)+\phi'(r)=0
\qquad
\text{for } r\in(0,\rho),
$$
that is, $(r\phi'(r))'=0$ on $(0,\rho)$. Thus $r\phi'(r)=C$ for some constant $C$. Since
$\phi$ is smooth on $[0,\rho)$, the function $\phi'$ is bounded near $0$, and this forces
$C=0$. Hence $\phi'(r)=0$ on $(0,\rho)$, so $\phi$ is constant on $(0,\rho)$, and therefore
on $[0,\rho)$ by continuity.

Thus $u$ is constant.
\end{proof}

\begin{prop}\label{ode}
Assume that $g_{\W_{m,s}}$ is K\"ahler--Einstein. Then there exists a positive constant $C_6$
such that
\begin{equation}\label{odeq}
\bigl(A+s\,t\,p(t)\bigr)^n
p(t)^{m-1}\bigl(p(t)+t\,p'(t)\bigr)
=
C_6R(t).
\end{equation}
\end{prop}

\begin{proof}
Fix $z_0\in\W$. By Lemma~\ref{det},
\[
\det(g_{\alpha\ov\beta})(z_0,\zeta)
=
C_1K_\W(z_0,\ov z_0)^A
\bigl(A+s\,t\,p(t)\bigr)^n
p(t)^{m-1}\bigl(p(t)+t\,p'(t)\bigr),
\]
where $t=K_\W(z_0,\ov z_0)^s\|\zeta\|^2$. By \eqref{localein2}, the function
\[
u(z,\zeta):=\log\frac{\det(g_{\alpha\bar\beta})(z,\zeta)}
                          {K_{\W_{m,s}}(z,\zeta)}
\]
is pluriharmonic in the standard coordinates of $\W_{m,s}$. On the fiber over $z_0$,
the determinant formula above and \eqref{kR} show that $u(z_0,\zeta)$ depends only on
$t=K_\W(z_0,\bar z_0)^s\|\zeta\|^2$. This fiber is a ball centered at the origin,
so Lemma~\ref{lem:radialpluri} shows that $u(z_0,\zeta)$ is constant there.
Hence there exists a positive constant $C_6(z_0)$ such that
\[
\bigl(A+s\,t\,p(t)\bigr)^n p(t)^{m-1}\bigl(p(t)+t\,p'(t)\bigr)=C_6(z_0)R(t).
\]

Evaluating at $t=0$, we obtain
\[
A^n p(0)^m=C_6(z_0)R(0).
\]
Since $A$, $p(0)$ and $R(0)$ are independent of $z_0$, it follows that $C_6(z_0)$ is actually
independent of $z_0$. This proves \eqref{odeq}.
\end{proof}

\subsection{Collapse of the radial factor}

\begin{thm}\label{coll}
Assume that $g_{\W_{m,s}}$ is K\"ahler--Einstein and that $s>0$. Then there exists
$c>0$ such that, with $R(t)$ as defined in \eqref{Rdef},
\[
R(t)=c(1-t)^{-(n+m+1)}.
\]
\end{thm}

\begin{proof}
By Proposition~\ref{ode}, there exists $C_6>0$ such that
\begin{equation}\label{odeq-coll}
\bigl(A+s\,t\,p(t)\bigr)^n
p(t)^{m-1}\bigl(p(t)+t\,p'(t)\bigr)
=
C_6R(t).
\end{equation}
Set
\[
y:=\frac{1}{1-t}.
\]
Since
\[
R(t)=\sum_{j=0}^n\frac{A_j}{(1-t)^{j+m+1}},
\]
we may write
\begin{equation}\label{Ry-coll}
R(t)=y^{m+1}P(y),
\qquad
P(y):=\sum_{j=0}^nA_jy^j.
\end{equation}
Because $s>0$, the polynomial $F_\W(sx)$ appearing in the expansion \eqref{poch}
 has degree $n$, hence $c_n(s)\neq0$ and therefore
$A_n\neq0$. Thus $P$ has degree exactly $n$. Both sides of \eqref{odeq-coll} are rational
functions of $y$. Since they agree for $1<y<\infty$, they agree as rational functions.

We claim that $P$ has no zero in $\C\setminus\{0\}$. Assume by contradiction that
$P(\alpha)=0$ for some $\alpha\in\C\setminus\{0\}$. Let $q\ge1$ be the multiplicity of
$\alpha$. Then
\[
P(y)=(y-\alpha)^qQ(y),
\qquad
Q(\alpha)\neq0.
\]
Hence
\[
\frac{P'(y)}{P(y)}=\frac{q}{y-\alpha}+\frac{Q'(y)}{Q(y)}.
\]
Now
\[
p(t)=\frac{d}{dt}\log R(t)
=
\frac{dy}{dt}\,\frac{d}{dy}\bigl((m+1)\log y+\log P(y)\bigr),
\]
and since $dy/dt=y^2$, we obtain
\[
p(t)=(m+1)y+y^2\frac{P'(y)}{P(y)}.
\]
Using $y^2=\alpha^2+(y-\alpha)(y+\alpha)$, we rewrite this as
\begin{equation}\label{ppole-coll}
p(t)=\frac{q\alpha^2}{y-\alpha}+G(y),
\end{equation}
where $G(y)$ is holomorphic near $\alpha$. Thus $p(t)$ has a pole of order $1$ at $y=\alpha$.

Differentiating with respect to $y$ and using again $dy/dt=y^2$, we get
\begin{equation}\label{pprime-coll}
p'(t)
=
-\frac{q\alpha^4}{(y-\alpha)^2}+\frac{L(y)}{y-\alpha},
\end{equation}
with $L(y)$ holomorphic near $\alpha$.

Since $t=1-\frac1y$ and $\alpha\neq0$, the function $t(y)$ is holomorphic at $y=\alpha$.
Moreover, since the Bergman kernel is positive on the diagonal, \eqref{kR} gives $R(0)>0$; hence
\[
R(0)=P(1)>0, 
\]
so $P(1)\neq0$, hence $\alpha\neq1$. Therefore
\[
t(\alpha)=1-\frac1\alpha\neq0.
\]
Write
\[
t(y)=t(\alpha)+(y-\alpha)M(y)
\]
with $M$ holomorphic near $\alpha$. Using \eqref{ppole-coll} and \eqref{pprime-coll}, we obtain
\[
p(t)+t\,p'(t)
=
-\frac{q\alpha^4t(\alpha)}{(y-\alpha)^2}+\frac{H(y)}{y-\alpha},
\]
where $H(y)$ is holomorphic near $\alpha$. Since $q\ge1$, $\alpha\neq0$, and $t(\alpha)\neq0$,
the coefficient of $(y-\alpha)^{-2}$ is nonzero. Hence $p(t)+t\,p'(t)$ has a pole of order $2$
at $y=\alpha$.

The factor $p(t)^{m-1}$ contributes $m-1$ to the pole order (zero when $m=1$).
Since $s>0$, the factor
$A+s\,t\,p(t)$ has a pole of order $1$, so its $n$-th power has a pole of order $n$. Therefore
the left-hand side of \eqref{odeq-coll} has a pole of order $n+m+1$ at $y=\alpha$.

On the other hand, by \eqref{Ry-coll},
\[
R(t)=y^{m+1}P(y),
\]
and since $\alpha\neq0$, the factor $y^{m+1}$ is holomorphic and nonzero at $y=\alpha$.
Hence $R(t)$ is holomorphic at $y=\alpha$, a contradiction. Therefore $P$ has no zero in
$\C\setminus\{0\}$.

Since $P$ has degree $n$ and all its zeros are equal to $0$, it follows that
\[
P(y)=c\,y^n
\]
for some $c\neq0$. Since $R(0)=P(1)>0$, one has $c>0$. Finally,
\[
R(t)=y^{m+1}P(y)=c\,y^{n+m+1}=c(1-t)^{-(n+m+1)}.
\]
\end{proof}

\subsection{Determination of the parameter and the base}


\begin{lem}\label{poly-sval}
Assume that \(R(t)=c(1-t)^{-(n+m+1)}\) for some \(c>0\). Then
\begin{equation}\label{topterm}
F_\W(sx)=\frac{1}{n!}(x+1)_n,
\end{equation}
and, in particular, \(s=\frac{1}{n+1}\).
\end{lem}

\begin{proof}
By definition,
\[
R(t)=\sum_{j=0}^n\frac{A_j}{(1-t)^{j+m+1}},
\qquad A_j=c_j(s)(j+m)!.
\]
Since the functions \(\{(1-t)^{-(j+m+1)}\}_{j=0}^n\) are linearly independent, the identity
\(R(t)=c(1-t)^{-(n+m+1)}\) implies that \(A_0=\cdots=A_{n-1}=0\). Hence
\(c_0(s)=\cdots=c_{n-1}(s)=0\), and by \eqref{poch} one gets
\(F_\W(sx)=c_n(s)(x+1)_n\). Evaluating at \(x=0\) and using \(F_\W(0)=1\), we find
\(c_n(s)=1/n!\), which proves \eqref{topterm}.

Now \(\frac{1}{n!}(x+1)_n=\prod_{j=1}^n\left(1+\frac{x}{j}\right)\), while by definition
\[
F_\W(sx)=\prod_{k=1}^r\prod_{i=1}^{1+p_k+b_k}\left(1+\frac{s\,x}{a_{k,i}}\right).
\]
Comparing zeros with multiplicity, we obtain
\begin{equation}\label{multiset2}
\{a_{k,i}\}=\{s,2s,\dots,ns\}
\qquad\text{with repetitions counted on both sides}.
\end{equation}

For each \(k\), write \(m_k:=1+p_k+b_k\) and
\(D_k:=2+p_k+q_k+b_k=m_k+1+q_k\). By \eqref{aki},
\(a_{k,i}=\frac{i+\frac{q_k}{2}}{D_k}\) for \(i=1,\dots,m_k\), so
\[
\sum_{i=1}^{m_k}a_{k,i}
=
\frac{1}{D_k}\sum_{i=1}^{m_k}\left(i+\frac{q_k}{2}\right)
=
\frac{m_k}{2}.
\]
Summing over \(k\) and using \eqref{deg}, we get
\[
\sum_{k=1}^r\sum_{i=1}^{m_k}a_{k,i}
=
\frac12\sum_{k=1}^r m_k
=
\frac n2.
\]
On the other hand, \eqref{multiset2} gives
\[
\sum_{k=1}^r\sum_{i=1}^{m_k}a_{k,i}
=
\sum_{j=1}^n js
=
s\frac{n(n+1)}{2}.
\]
Comparing the two identities yields \(s=\frac{1}{n+1}\).
\end{proof}

\begin{lem}\label{rank}
Assume that
\[
\{a_{k,i}\}
=
\left\{\frac{1}{n+1},\frac{2}{n+1},\dots,\frac{n}{n+1}\right\}
\]
as multisets. Then $\W$ has rank one. In particular,
\[
\W\cong\B^n.
\]
\end{lem}

\begin{proof}
For each fixed $k$, let $m_k:=1+p_k+b_k$. Then
\[
a_{k,i}=\frac{i+\frac{q_k}{2}}{m_k+1+q_k},
\qquad i=1,\dots,m_k,
\]
so
\[
a_{k,i+1}-a_{k,i}=\frac{1}{m_k+1+q_k}>0.
\]
Thus $a_{k,1},\dots,a_{k,m_k}$ is strictly increasing in $i$, and the minimum of the $k$-th
block is $a_{k,1}$.

Since the total collection of numbers, counted with repetitions, is
\[
\left\{\frac{1}{n+1},\frac{2}{n+1},\dots,\frac{n}{n+1}\right\},
\]
its minimum is $1/(n+1)$. Hence there exists $k_0$ such that
\[
a_{k_0,1}=\frac{1}{n+1}.
\]
Therefore
\[
\frac{1+\frac{q_{k_0}}{2}}{m_{k_0}+1+q_{k_0}}=\frac{1}{n+1},
\]
and hence
\begin{equation}\label{mkformula}
m_{k_0}=n+\frac{n-1}{2}q_{k_0}.
\end{equation}

If $n\ge2$, then $(n-1)/2>0$, so \eqref{mkformula} and $\sum_{k=1}^r m_k=n$ imply
\[
m_{k_0}=n,
\qquad q_{k_0}=0.
\]
Hence there is no block other than $k_0$, so $r=1$. If $n=1$, then
\[
\sum_{k=1}^r m_k=1,
\]
and since each $m_k\ge1$, again $r=1$. Thus $\W$ has rank one. By the standard realization of bounded homogeneous domains as homogeneous Siegel
domains (see \cite[pp.~220--221]{PyateskiiShapiro1969BOOKAutomorphicFunctions} and
\cite[Theorem~2.1(1)]{Seo2018BiholomorphismsHartogsSiegel}), \(\W\) is biholomorphically
equivalent to the domain associated with a \(j\)-algebra of rank one. Since any \(j\)-algebra of
rank one is elementary, and the domain corresponding to an elementary \(j\)-algebra is the unit
ball \cite[p.~53]{PyateskiiShapiro1969BOOKAutomorphicFunctions}, it follows that
$$
\W\cong \B^n.
$$

\end{proof}

\subsection{Conclusion of the proof}

\begin{proof}[Proof of Theorem~\ref{main}]
We prove the equivalence of the four conditions. First assume that the Bergman metric
\(g_{\W_{m,s}}\) of \(\W_{m,s}\) is K\"ahler--Einstein.
By Theorem~\ref{coll},
\[
R(t)=c(1-t)^{-(n+m+1)}
\]
for some \(c>0\). By Lemma~\ref{poly-sval}, we obtain
\[
s=\frac{1}{n+1}
\quad\text{and}\quad
F_\W(sx)=\frac{1}{n!}(x+1)_n.
\]
By the factor comparison in the proof of Lemma~\ref{poly-sval}, we obtain
\[
\{a_{k,i}\}
=
\left\{\frac{1}{n+1},\frac{2}{n+1},\dots,\frac{n}{n+1}\right\}
\]
with repetitions counted. By Lemma~\ref{rank}, we have \(\W\cong\B^n\).
This proves that condition \hyperref[main:ke]{{\rm(1)}} implies condition
\hyperref[main:base]{{\rm(4)}}.

We now prove that condition \hyperref[main:base]{{\rm(4)}} implies condition
\hyperref[main:ball]{{\rm(3)}}. Assume that
\(\W\cong\B^n\) and \(s=\frac{1}{n+1}\).
Choose a biholomorphism
\(\Phi:\W\to\B^n\). Since \(\det\Phi'(z)\neq 0\) on \(\W\) and \(\W\) is simply connected,
\(\det\Phi'\) admits a holomorphic logarithm on \(\W\). Therefore there exists a holomorphic
function \(h\) on \(\W\) such that \(|h(z)|^2=|\det\Phi'(z)|^{2s}\) for all \(z\in\W\). Define
\[
\widetilde\Phi:\W\times\C^m\longrightarrow \B^n\times\C^m,
\qquad
\widetilde\Phi(z,\zeta):=(\Phi(z),h(z)\zeta).
\]
Then \(\widetilde\Phi\) is biholomorphic.

If \(w=\Phi(z)\) and \(\eta=h(z)\zeta\), then the transformation law for the Bergman kernel gives
\(K_\W(z,\ov z)=K_{\B^n}(w,\ov w)\,|\det\Phi'(z)|^2\). Hence
\[
\|\eta\|^2=|h(z)|^2\|\zeta\|^2<|h(z)|^2K_\W(z,\ov z)^{-s}=K_{\B^n}(w,\ov w)^{-s}.
\]
Therefore \(\widetilde\Phi\) identifies \(\W_{m,s}\) bijectively with
\[
\left\{(w,\eta)\in\B^n\times\C^m:\ \|\eta\|^2<K_{\B^n}(w,\ov w)^{-s}\right\}.
\]
Indeed, the same computation applied to the inverse biholomorphism $\Phi^{-1}$ gives the reverse inclusion, and $\widetilde\Phi^{-1}(w,\eta)=(\Phi^{-1}(w),h(\Phi^{-1}(w))^{-1}\eta)$.

Since \(K_{\B^n}(w,\ov w)=\frac{n!}{\pi^n}(1-\|w\|^2)^{-(n+1)}\) and
\(s=\frac{1}{n+1}\), the defining inequality becomes
\[
\|\eta\|^2<\left(\frac{n!}{\pi^n}\right)^{-1/(n+1)}(1-\|w\|^2).
\]
If we set \(C:=\left(\frac{n!}{\pi^n}\right)^{-1/(n+1)}\), this is equivalent to
\[
\|w\|^2+\frac{1}{C}\|\eta\|^2<1.
\]
After the linear change of variable \(u=C^{-1/2}\eta\), we obtain
\[
\|w\|^2+\|u\|^2<1.
\]
Hence the image of \(\W_{m,s}\) is exactly \(\B^{n+m}\), and therefore
\[
\W_{m,s}\cong\B^{n+m}.
\]
Thus condition \hyperref[main:base]{{\rm(4)}} implies condition
\hyperref[main:ball]{{\rm(3)}}.

Condition \hyperref[main:ball]{{\rm(3)}} immediately implies condition
\hyperref[main:hom]{{\rm(2)}}, since the unit ball is homogeneous. Finally, condition
\hyperref[main:hom]{{\rm(2)}} implies condition \hyperref[main:ke]{{\rm(1)}} by the same
argument used in the proof of Lemma~\ref{baseke}: the Bergman kernel and the determinant of the
Bergman metric have the same transformation law under biholomorphisms, so the ratio
$
\frac{\det(g_{\alpha\ov\beta})}{K_{\W_{m,s}}}
$
is invariant under \(\Aut(\W_{m,s})\). Since \(\W_{m,s}\) is homogeneous, this ratio is constant.
Thus
\[
\Ric(g_{\W_{m,s}})
=
-\frac{i}{2}\de\deb\log\det(g_{\alpha\ov\beta})
=
-\frac{i}{2}\de\deb\log K_{\W_{m,s}}
=
-\omega_{\W_{m,s}},
\]
and the Bergman metric of \(\W_{m,s}\) is K\"ahler--Einstein. This proves condition
\hyperref[main:ke]{{\rm(1)}} and completes the proof.
\end{proof}

\providecommand{\etalchar}[1]{$^{#1}$}


\begin{thebibliography}{WYZR06}
\providecommand{\url}[1]{#1}
\csname url@samestyle\endcsname
\providecommand{\newblock}{\relax}
\providecommand{\bibinfo}[2]{#2}
\providecommand{\BIBentrySTDinterwordspacing}{\spaceskip=0pt\relax}
\providecommand{\BIBentryALTinterwordstretchfactor}{4}
\providecommand{\BIBentryALTinterwordspacing}{\spaceskip=\fontdimen2\font plus
\BIBentryALTinterwordstretchfactor\fontdimen3\font minus
  \fontdimen4\font\relax}
\providecommand{\BIBforeignlanguage}[2]{{%
\expandafter\ifx\csname l@#1\endcsname\relax
\typeout{** WARNING: IEEEtranSA.bst: No hyphenation pattern has been}%
\typeout{** loaded for the language `#1'. Using the pattern for}%
\typeout{** the default language instead.}%
\else
\language=\csname l@#1\endcsname
\fi
#2}}
\providecommand{\BIBdecl}{\relax}
\BIBdecl

\bibitem[ABP12]{AhnByunPark2012AutomorphismsHartogsDomain}
\BIBentryALTinterwordspacing
H.~Ahn, J.~Byun, and J.-D. Park, ``Automorphisms of the {H}artogs type domains
  over classical symmetric domains'', \emph{Internat. J. Math.}, vol.~23,
  no.~9, Paper No.~1250098, 11~pp., 2012.
  \url{https://doi.org/10.1142/S0129167X1250098X}
\BIBentrySTDinterwordspacing

\bibitem[BP98]{BlockiPflug98}
Z.~B\l{}ocki and P.~Pflug, ``Hyperconvexity and Bergman completeness'',
  \emph{Nagoya Math. J.}, vol.~151, pp.~221--225, 1998.
  \url{https://doi.org/10.1017/S0027763000025265}

\bibitem[Cal53]{Calabi1953IsometricImbedding}
\BIBentryALTinterwordspacing
E.~Calabi, ``Isometric imbedding of complex manifolds'', \emph{Ann. of Math.
  (2)}, vol.~58, pp. 1--23, 1953. \url{https://doi.org/10.2307/1969817}
\BIBentrySTDinterwordspacing

\bibitem[Cho21]{Cho21}
G.~Cho, ``Invariant metrics on the complex ellipsoid'', \emph{J. Geom. Anal.},
  vol.~31, no.~2, pp.~2088--2104, 2021.
  \url{https://doi.org/10.1007/s12220-019-00333-w}

\bibitem[EXX23]{EXX23}
P.~Ebenfelt, M.~Xiao, and H.~Xu, ``K\"ahler--Einstein metrics and obstruction
  flatness of circle bundles'', \emph{J. Math. Pures Appl.}, vol.~177,
  pp.~368--414, 2023. \url{https://doi.org/10.1016/j.matpur.2023.07.003}

\bibitem[EXX24]{EXX24}
P.~Ebenfelt, M.~Xiao, and H.~Xu, ``K\"ahler--Einstein metrics and obstruction
  flatness II: Unit sphere bundles'', \emph{J. Funct. Anal.}, vol.~286,
  no.~9, Paper No.~110366, 2024. \url{https://doi.org/10.1016/j.jfa.2024.110366}

\bibitem[FW97]{FuWong2005StrictlyPseudoconvexDomains}
\BIBentryALTinterwordspacing
S.~Fu and B.~Wong, ``On strictly pseudoconvex domains with
  {K}\"{a}hler-{E}instein {B}ergman metrics'', \emph{Math. Res. Lett.}, vol.~4,
  no.~5, pp. 697--703, 1997. \url{https://doi.org/10.4310/MRL.1997.v4.n5.a7}
\BIBentrySTDinterwordspacing

\bibitem[HX21]{HuangXiao2021BergmanEinstein}
\BIBentryALTinterwordspacing
X.~Huang and M.~Xiao, ``Bergman-{E}instein metrics, a generalization of
  {K}erner's theorem and {S}tein spaces with spherical boundaries'', \emph{J.
  Reine Angew. Math.}, vol. 770, pp. 183--203, 2021.
  \url{https://doi.org/10.1515/crelle-2020-0012}
\BIBentrySTDinterwordspacing

\bibitem[IPY17]{IshiParkYamamori2017BergmanKernel}
\BIBentryALTinterwordspacing
H.~Ishi, J.-D. Park, and A.~Yamamori, ``Bergman kernel function for {H}artogs
  domains over bounded homogeneous domains'', \emph{J. Geom. Anal.}, vol.~27,
  no.~2, pp. 1703--1736, 2017. \url{https://doi.org/10.1007/s12220-016-9737-4}
\BIBentrySTDinterwordspacing

\bibitem[KO07]{KaiOhsawa07}
C.~Kai and T.~Ohsawa, ``A note on the Bergman metric of bounded homogeneous domains'',
  \emph{Nagoya Math. J.}, vol.~186, pp.~157--163, 2007.
  \url{https://doi.org/10.1017/S0027763000009399}

\bibitem[Kan71]{Kaneyuki1971BOOKHomogeneousBoundedDomains}
\BIBentryALTinterwordspacing
S.~Kaneyuki, \emph{Homogeneous bounded domains and {S}iegel domains}, ser.
  Lecture Notes in Mathematics.\hskip 1em plus 0.5em minus 0.4em\relax
  Springer-Verlag, Berlin-New York, 1971, vol. 241.
  \url{https://doi.org/10.1007/BFb0060967}
\BIBentrySTDinterwordspacing

\bibitem[Kob59]{Kobayashi1959GeometryBounded}
\BIBentryALTinterwordspacing
S.~Kobayashi, ``Geometry of bounded domains'', \emph{Trans. Amer. Math. Soc.},
  vol.~92, pp. 267--290, 1959. \url{https://doi.org/10.2307/1993156}
\BIBentrySTDinterwordspacing

\bibitem[LM23]{LoiMossa2023HolomorphicIsometries}
\BIBentryALTinterwordspacing
A.~Loi and R.~Mossa, ``Holomorphic isometries into homogeneous bounded
  domains'', \emph{Proc. Amer. Math. Soc.}, vol. 151, no.~9, pp. 3975--3984,
  2023. \url{https://doi.org/10.1090/proc/16335}
\BIBentrySTDinterwordspacing

\bibitem[LMZ25]{LoiMossaZuddas2025BergmanMetric}
\BIBentryALTinterwordspacing
A.~Loi, R.~Mossa, and F.~Zuddas, ``On the Bergman metric of Cartan--Hartogs
  domains'', 2025. \url{https://arxiv.org/abs/2510.06405}
\BIBentrySTDinterwordspacing

\bibitem[MZ22]{MossaZedda2022SymplecticGeometry}
\BIBentryALTinterwordspacing
R.~Mossa and M.~Zedda, ``Symplectic geometry of {C}artan-{H}artogs domains'',
  \emph{Ann. Mat. Pura Appl. (4)}, vol. 201, no.~5, pp. 2315--2339, 2022.
  \url{https://doi.org/10.1007/s10231-022-01201-1}
\BIBentrySTDinterwordspacing

\bibitem[NS06]{NemirovskiiShafikov2006ConjectureCheng}
\BIBentryALTinterwordspacing
S.~Y. Nemirovski\u{\i} and R.~G. Shafikov, ``Conjectures of {C}heng and
  {R}amadanov'', \emph{Uspekhi Mat. Nauk}, vol.~61, no. 4(370), pp. 193--194,
  2006. \url{https://doi.org/10.1070/RM2006v061n04ABEH004349}
\BIBentrySTDinterwordspacing

\bibitem[Pal25]{Palmieri2025BergmanMetricsInducedBall}
\BIBentryALTinterwordspacing
M.~Palmieri, ``Bergman metrics induced by the ball'', 2025.
  \url{https://arxiv.org/abs/2510.17618}
\BIBentrySTDinterwordspacing

\bibitem[PS69]{PyateskiiShapiro1969BOOKAutomorphicFunctions}
I.~I. Pyateskii-Shapiro, \emph{Automorphic functions and the geometry of
  classical domains}, ser. Mathematics and its Applications, Vol. 8.\hskip 1em
  plus 0.5em minus 0.4em\relax Gordon and Breach Science Publishers, New
  York-London-Paris, 1969, translated from the Russian.

\bibitem[PWZ16]{PanWangZhang2016KahlerEinstein}
\BIBentryALTinterwordspacing
L.~Pan, A.~Wang, and L.~Zhang, ``On the {K{\"a}hler--Einstein} metric of
  {B}ergman--{H}artogs domains'', \emph{Nagoya Math. J.}, vol. 221, no.~1, pp.
  184--206, 2016. \url{https://doi.org/10.1017/nmj.2016.4}
\BIBentrySTDinterwordspacing

\bibitem[Ros79]{Rosay1979SurUneCaracterisation}
\BIBentryALTinterwordspacing
J.-P. Rosay, ``Sur une caract\'erisation de la boule parmi les domaines de
  {${\bf C}\sp{n}$}\ par son groupe d'automorphismes'', \emph{Ann. Inst.
  Fourier (Grenoble)}, vol.~29, no.~4, pp. ix, 91--97, 1979.
  \url{https://doi.org/10.5802/aif.768}
\BIBentrySTDinterwordspacing

\bibitem[Seo18]{Seo2018BiholomorphismsHartogsSiegel}
\BIBentryALTinterwordspacing
A.~Seo, ``Biholomorphisms between {H}artogs domains over homogeneous {S}iegel
  domains'', \emph{Internat. J. Math.}, vol.~29, no.~8, Paper No.~1850057, 12~pp.,
  2018. \url{https://doi.org/10.1142/S0129167X1850057X}
\BIBentrySTDinterwordspacing

\bibitem[Sha26]{Sha26}
Z.~Sha, ``The K\"ahler--Ricci soliton on bounded pseudoconvex domains'',
  arXiv:2412.03345v5, 2026, to appear in \emph{Bull. Sci. Math.}
  \url{https://arxiv.org/abs/2412.03345v5}

\bibitem[Won77]{Wong1977CharacterizationUnitBall}
\BIBentryALTinterwordspacing
B.~Wong, ``Characterization of the unit ball in {${\bf C}\sp{n}$} by its
  automorphism group'', \emph{Invent. Math.}, vol.~41, no.~3, pp. 253--257,
  1977. \url{https://doi.org/10.1007/BF01403050}
\BIBentrySTDinterwordspacing

\bibitem[WYZR06]{WangYinZhangRoos2006KahlerEinstein}
\BIBentryALTinterwordspacing
A.~Wang, W.~Yin, L.~Zhang, and G.~Roos, ``The {K}\"{a}hler-{E}instein metric
  for some {H}artogs domains over symmetric domains'', \emph{Sci. China Ser.
  A}, vol.~49, no.~9, pp. 1175--1210, 2006.
  \url{https://doi.org/10.1007/s11425-006-0230-6}
\BIBentrySTDinterwordspacing

\bibitem[Yau82]{Yau1982ProblemSection}
S.~T. Yau, ``Problem section'', in \emph{Seminar on {D}ifferential {G}eometry},
  ser. Ann. of Math. Stud., No. 102.\hskip 1em plus 0.5em minus 0.4em\relax
  Princeton Univ. Press, Princeton, NJ, 1982, pp. 669--706.

\bibitem[YLR04]{YinLuRoos2004NewClasses}
\BIBentryALTinterwordspacing
W.~Yin, K.~Lu, and G.~Roos, ``New classes of domains with explicit {B}ergman
  kernel'', \emph{Sci. China Ser. A}, vol.~47, no.~3, pp. 352--371, 2004.
  \url{https://doi.org/10.1360/03ys0090}
\BIBentrySTDinterwordspacing

\end{thebibliography}
\end{document}